\documentclass{conm-p-l}

\usepackage{tikz}
\usepackage{float}
\copyrightinfo{2015}{Johan P. Hansen}
\newtheorem{theorem}{Theorem}[section]
\newtheorem{lemma}[theorem]{Lemma}
\newtheorem{cor}[theorem]{Corollary}
\newtheorem{prop}[theorem]{Proposition}
\theoremstyle{definition}
\newtheorem{definition}[theorem]{Definition}

\theoremstyle{remark}
\newtheorem{remark}[theorem]{Remark}

\numberwithin{equation}{section}

\newcommand{\F}{\mathbb F_q^*}
\newcommand{\Fq}{\mathbb F_q}
\newcommand{\Z}{\mathbb Z}
\newcommand{\Fbarq}{{\overline{\mathbb F}_q}}
\newcommand{\refdefn}[1]{Definition~\ref{#1}}
\newcommand{\refprop}[1]{Proposition~\ref{#1}}
\newcommand{\reftheorem}[1]{Theorem~\ref{#1}}
\newcommand{\reflemma}[1]{Lemma~\ref{#1}}
\newcommand{\refcor}[1]{Corollary~\ref{#1}}
\newcommand{\Hom}{\operatorname{Hom}}
\newcommand{\orb}{\operatorname{orb}}
\begin{document}

\title[Secret Sharing on Toric Varieties]{Secret Sharing Schemes with Strong Multiplication and a Large Number of Players from Toric Varieties}

\author[Johan P. Hansen]{Johan P. Hansen}

\address{Department of Mathematics, Aarhus University, Ny Munkegade 118, DK-8000 Aarhus C, Denmark}

\email{matjph@math.au.dk}
\thanks{Part of this work was done while visiting Institut de Math\'ematiques de Luminy, MARSEILLE, France.
I thank for the hospitality shown to me.}
\thanks{This work was supported by the Danish Council for Independent
Research, grant no. DFF-4002-00367.}


\date{1. March 2016}
\begin{abstract}
This article consider Massey's construction for constructing linear secret sharing schemes from  toric varieties over a finite field $\Fq$ with $q$ elements. The number of players can be as large as $(q-1)^r-1$ for $r\geq 1$. The schemes have strong multiplication, such schemes can be utilized in the domain of multiparty computation.

We present general methods to obtain the reconstruction and privacy thresholds as well as conditions for multiplication on the associated secret sharing schemes.

In particular we apply the method on certain toric surfaces. The main results are ideal linear secret sharing schemes where the number of players can be as large as $(q-1)^2-1$, we determine bounds for the reconstruction and privacy thresholds and conditions for strong multiplication using the cohomology and the intersection theory on toric surfaces.
\end{abstract}

\maketitle
\tableofcontents

\subsection*{Notation}
\begin{itemize}
\item $\mathbb F_q$ -- the finite field with $q$ elements of characteristic $p$.
\item $\F$ -- the invertible elements in $\mathbb F_q$.
\item $k=\overline{\mathbb F_q}$ -- an algebraic closure of $\mathbb F_q$.
\item $M \simeq \mathbb Z^r$ a free $\mathbb Z$-module of rank r.
\item $\square \subseteq M_{\mathbb R}= M \otimes_{\mathbb Z}{\mathbb R}$ -- an integral convex polytope.
\item $X=X_{\square}$ -- the toric variety associated to the polytope $\square$.
\item $T=T_N =U_0 \subseteq X$ -- the torus.
\item $H =\{0,1,\dots,q-2\}\times \dots \times \{0,1,\dots,q-2\} \subset M$.

\end{itemize}

\section{Introduction}

\subsection{Secret sharing}
Secret sharing schemes were introduced in \cite{Blakley1979} and
\cite{journals/cacm/Shamir79} and provide a method to split a \emph{secret} into several pieces of information (\emph{shares}) so that any large enough subset of the shares determines the secret, while any small subset of shares provides no information on the secret. 

Secret sharing schemes have found applications in cryptography, when the schemes has certain algebraic properties. \emph{Linear secret sharing schemes} (LSSS) are schemes where the secrets $s$  and their associated shares $(a_1,\dots, a_n)$ are elements in a vector space over some finite ground field $\Fq$. The schemes are  called \emph{ideal} if the secret $s$ and the shares $a_i$ are elements in that ground field $\Fq$. Specifically, if $s, \tilde{s} \in \Fq$ are two secrets with share vectors $(a_1,\dots a_n), (\tilde{a}_1,\dots \tilde{a}_n) \in \Fq^n$, then the share vector of the secret $s+\lambda \tilde{s} \in \Fq$ is $(a_1+\lambda \tilde{a}_1,\dots, a_n+ \lambda \tilde{a}_n) \in \Fq^n$ for any $\lambda \in \Fq$.

\emph{The reconstruction threshold} of the linear secret sharing scheme is the smallest integer $r$ so that any set of at least $r$ of the shares $a_1,\dots,a_n$ determines the secret $s$. The \emph{privacy threshold} is the largest integer $t$ such that no set of $t$ (or fewer) elements of the shares $a_1,\dots,a_n$ determines the secret $s$. The scheme is said to have $t$-\emph{privacy}.

An ideal linear secret sharing scheme is said to have \emph{multiplication} if the product of the shares determines the product of the secrets. It has $t$-\emph{strong multiplication} if it has $t$-\emph{privacy} and 
has multiplication for any subset of $n-t$ shares obtained by removing any  $t$ shares.

The properties of multiplication was introduced in \cite{CDM}. Such schemes with multiplication can be utilized in the domain of multiparty computation (MPC), see \cite{DBLP:conf/stoc/ChaumCD88}, \cite{BenOR}, \cite{17425} and \cite{DBLP:books/cu/CDN2015}.

\subsection{Toric varieties and secret sharing}
In  \cite{6ffeb030f4f511dd8f9a000ea68e967b}, \cite{39bd8e90f4f211dd8f9a000ea68e967b} and \cite{53ee7c6020b511dcbee902004c4f4f50}  we developed methods to construct linear error correcting codes from toric varieties and derived the code parameters using the cohomology and the intersection theory on toric varieties. In \cite{ffdba1be1ed7444aa8ec4c25ad8a10bb} we utilized the method and the results to construct quantum codes.

Massey's construction of linear secret sharing schemes from error-correcting codes \cite{MR2017562} also applies to our codes on toric varieties. In a certain sense our construction resembles that of \cite{CC}, where LSSS schemes were constructed from Goppa codes on algebraic curves, however, the methods to obtain the parameters are completely different. 

The linear secret sharing schemes we obtain are \emph{ideal} and the number of players  are $q^r-1$ for any positive integer $r$. The classical Shamir scheme only allows $q-1$ players, however, there are methods to allow schemes with more players using linear codes \cite{MR2449216}, this article presents such a method.

The schemes are obtained by evaluating certain rational functions in $\Fq$-rational points on toric varieties. 

The thresholds and conditions for strong multiplication are derived from estimates on the maximum number of zeroes of rational functions obtained via the cohomology and intersection theory on the underlying toric variety. In particular, we focus on toric surfaces.

We present examples of linear secret sharing schemes which are \emph{quasi-threshold} and have \emph{strong multiplication} \cite{CDM} with respect to certain adversary structures.

Specifically, for any pair of integers $a,b$, with $0 \leq b \leq a \leq q-2$, we produce linear secret sharing schemes with $(q-1)^2-1$ players which are \emph{quasi-threshold}, i.e., the reconstruction threshold is at most $1+ (q-1)^2-(q-1-a)$ and the privacy threshold is at least $b-1$.
The schemes have $t$-\emph{strong multiplication} with respect to the threshold adversary structure if 
$t\leq \min\{b-1,(q-2-2a)-1\}$.

For the general theory of toric varieties, we refer to \cite{Oda1988}, \cite{Fulton:1436535} and \cite{MR2810322}.

\section{Preliminaries}
\subsection{Linear Secret Sharing Schemes}

This section presents basic definitions and concepts pertaining to linear secret sharing schemes as introduced in \cite{MR2017562},\cite{CDM},
\cite{CC} and \cite{MR2449216}.

Let be $\Fq$  be a finite field with $q$ elements.

An \emph{ideal linear secret sharing scheme} $\mathcal M$ over a finite field $\Fq$ on a set $\mathcal{P}$ of $n$ players is given by a positive integer $e$, a sequence $V_1,\dots V_n$ of 1-dimensional linear subspaces $V_i \subset \Fq^e$ and a non-zero vector $u \in \Fq^e$.

An \emph{adversary structure} $\mathcal{A}$, for a  secret sharing scheme $\mathcal{M}$ on the set of players $\mathcal{P}$, is a collection of subsets of $\mathcal{P}$, with the property that subsets of sets in $\mathcal{A}$ are also sets in $\mathcal{A}$. In particular, the \emph{adversary structure} $\mathcal{A}_{t,n}$ consists of all the subsets of size at most $t$ of the set $\mathcal{P}$ of $n$ players, and the \emph{access structure} $\Gamma_{r,n}$ consists of all the subsets of size at least $r$ of the set $\mathcal{P}$ of $n$ players.

For any subset $A$ of players, let $V_A= \sum_{i\in A} V_i$ be the $\Fq$-subspace spanned by all the $V_i$ for $i \in A$. 

The \emph{access structure} $\Gamma(\mathcal{M})$ of $\mathcal{M}$ consists of all the subsets $B$ of players with $u \in V_B$, and $\mathcal{A}(\mathcal{M})$ consists of all the other subsets $A$ of players, that is $A \notin \Gamma(\mathcal{M})$.

A linear secret sharing scheme $\mathcal M$ is said to \emph{reject} a given adversary structure $\mathcal{A}$, if $\mathcal{A} \subseteq \mathcal{A}(\mathcal{M})$. Therefore $A \in \mathcal{A}(\mathcal{M})$ if and only if there is a linear map from $\Fq^e$ to $\Fq$ vanishing on $V_A$, while non-zero on $u$.

The scheme $\mathcal M$ works as follows. For $i = 1,\dots n$, let $v_i \in V_i$  be bases for the 1-dimensional vector spaces. Let $s \in \Fq$ be a \emph{secret}. Choose at random a linear morphism $\phi : 
\Fq^e \rightarrow \Fq$, subject to the condition $\phi(u)=s$, and let $a_i = \phi(v_i)$ for $i = 1,\dots, n$ be the \emph{shares}
\begin{eqnarray*}
\phi : \Fq^e& \rightarrow& \Fq \\
u& \mapsto& s\\
v_i& \mapsto & a_i\quad \mathrm{for\ } i = 1,\dots, n
\end{eqnarray*}
Then
\begin{itemize}
\item[-]the shares $\{a_i=\phi(v_i)\}_{i \in A}$ determine the secret $s=\phi(u)$ uniquely if and only if $A \in \Gamma(\mathcal{M})$,
\item[-]the shares $\{a_i=\phi(v_i)\}_{i \in A}$ reveal no information on the secret $s=\phi(u)$, i.e., when $A \in \mathcal{A}(\mathcal{M})$.
\end{itemize}

\begin{definition} \label{defthresholds}
Let $\mathcal M$ be a linear secret sharing scheme.

The \emph{reconstruction threshold} of  $\mathcal M$ is the smallest integer $r$ so that any set of at least $r$ of the shares $a_1,\dots,a_n$ determines the secret $s$, i.e., $\Gamma_{r,n} \subseteq \Gamma(\mathcal{M})$.

 The \emph{privacy threshold} is the largest integer $t$ so that no set of $t$ (or less) elements of the shares $a_1,\dots,a_n$ determine the secret $s$, i.e., $\mathcal{A}_{t,n}\subseteq \mathcal{A}(\mathcal{M})$.
The scheme $\mathcal M$ is said to have $t$-\emph{privacy}.
\end{definition}

\begin{definition}\label{defstrong}
An ideal linear secret sharing scheme $\mathcal{M}$ has the \emph{strong multiplication property} with respect to an adversary structure $\mathcal{A}$ if the following holds.
\begin{itemize}
\item[1.] $\mathcal{M}$ rejects the adversary structure $\mathcal{A}$\ .
\item[2.] Given two secrets $s$ and $\tilde{s}$. For each $A \in \mathcal{A}$, the products $a_i \cdot \tilde{a}_i$ of all the shares of the players $i \notin A$ determine the product $s \cdot \tilde{s}$ of the two secrets.
\end{itemize}
\end{definition}

\section{Linear secret sharing schemes with multiplication on tori}

In \cite{6ffeb030f4f511dd8f9a000ea68e967b}, \cite{39bd8e90f4f211dd8f9a000ea68e967b} and \cite{53ee7c6020b511dcbee902004c4f4f50} we introduced linear codes from toric varieties and estimated the minimum distance of such codes using intersection theory. Our method to estimate the minimum distance of toric codes has subsequently been supplemented, e.g., \cite{MR2272243}, \cite{MR2476837}, \cite{MR2322944}, \cite{MR2360532}, \cite{Beelen},\cite{MR3093852} \cite{MR3345095},  and \cite{DBLP:journals/corr/Little15}.

Linear secret sharing schemes obtained from linear codes were introduced by James L. Massey in \cite{MR2017562} and were generalized in \cite[Section 4.1]{MR2449216}. A scheme with $n$ players is obtained from a linear $C$ code of length $n+1$ and dimension $k$ with privacy threshold $t = d'-2$ and reconstruction threshold $r= n-d+2$, where $d$ is the minimum distance of the code and $d'$ the minimum distance of the dual code.

We utilize the Massey construction to obtain linear secret sharing schemes from toric codes. 

Under certain conditions the linear secret sharing schemes from toric codes have the strong multiplication property.

\subsection{The construction}
Let $M \simeq \Z^r$ be a free $\Z$-module of rank $r$ over the integers $\Z$. 

For any subset $U \subseteq M$, let
$\Fq<U>$ be the linear span in $\Fq[X_1^{\pm 1},\dots, X_r^{\pm 1}]$ of the monomials
\begin{equation*}
\{X^u=X_1^{u_1}\cdot \dots \cdot X_r^{u_r} \vert \ u=(u_1,\dots,u_r)\in U\}\ .
\end{equation*}
This is a $\Fq$-vector space of dimension equal to the number of elements in $U$.

Let $T(\Fq)=(\F)^r$ be the $\Fq$-rational points on the torus and let $S \subseteq T(\Fq)$ be any subset. The linear map that evaluates elements in $\Fq<U>$ at all the points in $S$ is denoted by $\pi_S$:
\begin{eqnarray*}
\pi_S:\Fq<U>&\rightarrow & \Fq^{\vert S\vert}\\
 f&\mapsto&(f(P))_{P\in S}\ .
\end{eqnarray*}
In this notation $\pi_{\{P\}}(f)=f(P)$. 

The toric code is the image $C=\pi_S(\Fq<U>)$ and we obtain a the linear secret sharing scheme from $C$ by the Massey construction.

\begin{definition}\label{LSSS}
Let $S \subseteq T(\Fq)$ be any subset so that $P_0 \in S$.
The linear secret sharing schemes (LSSS) $\mathcal{M}(U)$ with \emph{support} $S$ and $n=\vert S\vert-1$ players is obtained as follows:
\begin{itemize}
\item Let $s_0 \in \Fq$ be a \emph{secret} value. Select $f \in \Fq<U>$ at random, such that $\pi_{\{P_0\}}(f)=f(P_0)= s_0$.
\item Define the $n$ shares as 
\begin{equation*}
\pi_{S\setminus{\{P_0\}}}(f)= (f(P))_{P\in S\setminus{\{P_0\}}} \in \Fq^{\vert S\vert-1} = \Fq^n\ .
\end{equation*}
\end{itemize}
\end{definition}

The main objectives are to study \emph{privacy}, \emph{reconstruction} of the secret from the shares 
 and the property \emph{strong multiplication} of the scheme as introduced in \refdefn{defthresholds} and \refdefn{defstrong}.

In order to present the general theory for the linear secret sharing schemes $\mathcal{M}(U)$ above, we make some preliminary definitions and observations.

\subsubsection{Translation}
Let $U \subseteq M$ be a subset, let $v \in M$ and consider the translate $v + U:= \{v+u \vert \ u \in U \} \subseteq M$.
\begin{lemma}\label{trans}
Translation induces an isomorphism of vector spaces 
\begin{eqnarray*}
\Fq<U>&\rightarrow& \Fq<v+U>\\
f&\mapsto&f^v:=X^v \cdot f\ .
\end{eqnarray*}
We have that
\begin{itemize}
\item[\it i)]The evaluations of $\pi_{T(\Fq)}(f)$ and $\pi_{T(\Fq)}(f^v)$ have the same number of zeroes on $T(\Fq)$. 
\item[\it ii)]The minimal number of zeros on $T(\Fq)$ of evaluations of elements in $\Fq<U>$ and $\Fq<v+U>$ are the same.
\item[\it iii)]For $v=(v_1,\dots,v_r)$ with $v_i$ divisible by $q-1$, the evaluations $\pi_S(f)$ and $\pi_S(f^v)$ are the same for any subset $S$ of $T(\Fq)$.
\end{itemize}
\end{lemma}
The lemma and generalizations has been used in several articles classifying toric codes, e.g., \cite{MR2272243}.

An immediate consequence of {\it iii)} above is the following corollary, which also can be found in \cite[Theorem 3.3]{MR2360532}.
\begin{cor} \label{reduction}Let $U \subseteq M$ be a subset and let
\begin{equation*}
\bar{U}:=\{(\bar{u}_1,\dots,\bar{u}_r)\vert \  \bar{u}_i\in \{0,\dots,q-2\}\  \mathrm{and}\ \bar{u}_i \equiv u_i \mod q-1 \}
\end{equation*}
be its reduction modulo $q-1$.
Then $\pi_S(\Fq<U>) =\pi_S(\Fq<\bar{U}>)$  for any subset $S \subseteq T(\Fq)$.
\end{cor}

\subsubsection{Orthogonality - dual code}
In Proposition \ref{orto} we present the dual code of $C=\pi_S(\Fq<U>)$.

Let $U \subseteq M$ be a subset, define its opposite as $-U:= \{-u \vert \ u \in U \} \subseteq M$.
The opposite maps the monomial $X^u$ to $X^{-u}$ and induces by linearity an isomorphism of vector spaces 
\begin{eqnarray*}
\Fq<U>&\rightarrow& \Fq<-U>\\
X^u&\mapsto & X^{-u}\\
f&\mapsto&\hat{f}\ .
\end{eqnarray*}
On $\Fq^{\vert T(\Fq) \vert}$, we have the inner product 
\begin{equation*}
(a_0,\dots,a_n)\star (b_0,\dots,b_n)=\sum_{l=0}^n a_l b_l \in \Fq\ ,
\end{equation*}
with $n=\vert T(\Fq) \vert-1$.
\begin{lemma} Let $f,g \in \Fq <M> $ and assume $f \neq \hat{g}$, then
\begin{equation*}
\pi_{T(\Fq)}(f)\star \pi_{T(\Fq)}(g)= 0 
\end{equation*}
\end{lemma}

Let 
\begin{equation*}
H =\{0,1,\dots,q-2\}\times \dots \times \{0,1,\dots,q-2\} \subset M\ .
\end{equation*}

With this inner product we obtain the following proposition, e.g. \cite[Proposition 3.5]{MR2377234} and \cite[Theorem 6]{MR2499927}.
\begin{prop}\label{orto}
Let $U \subseteq H$ be a subset. Then we have
\begin{itemize}
\item[{\it i)}] For $f \in \Fq<U>$ and $g \notin \Fq<-H\setminus -U>$, we have that
$\pi_{T(\Fq)}(f)\star \pi_{T(\Fq)}(g)=0$. 

\item[{\it ii)}] The orthogonal complement to $\pi_{T(\Fq)}(\Fq<U>)$ in $\Fq^{\vert {T(\Fq)} \vert}$ is 
\begin{equation*}
\pi_{T(\Fq)}(\Fq<-H \setminus -U>)\ ,
\end{equation*}
i.e., the dual code of $C=\pi_{T(\Fq)}(\Fq<U>)$ is $\pi_{T(\Fq)}(\Fq<-H \setminus -U>)$.
\end{itemize}
\end{prop}

\begin{theorem} \label{thresholds} 
Let $r(U)$ and $t(U)$ be the reconstruction and privacy thresholds of $\mathcal{M}(U)$ as defined in \refdefn{defthresholds}. 

Then
\begin{eqnarray*}
r(U) &\geq & (\mathrm{the\ maximum\ number\ of\ zeros\ of\ } \pi_{T(\Fq)}(f))+2\\
t(U) & \leq & (q-1)^r-(\mathrm{the\ maximum\ number\ of\ zeros\ of\ } \pi_{T(\Fq)}(g))-2\ ,
\end{eqnarray*}
for some $f\in \Fq<U>$ and for some $g\in \Fq<-H\setminus -U>$\ ,
where
\begin{eqnarray*}
\pi_{T(\Fq)}:\Fq<U>&\rightarrow & \Fq^{\vert {T(\Fq)}\vert}\\
 f&\mapsto&\pi_{T(\Fq)}(f)=(f(P))_{P\in {T(\Fq)}}\\
\pi_{T(\Fq)}:\Fq<-H\setminus -U>&\rightarrow & \Fq^{\vert {T(\Fq)}\vert}\\
 g&\mapsto&\pi_{T(\Fq)}(g)=(g(P))_{P\in {T(\Fq)}}\ .
\end{eqnarray*}
\end{theorem}

\begin{proof}The minimal distance of an evaluation code and the maximum number of zeros of a function add to the length of the code. 

The bound for $r(U)$ is based on the minimum distance $d$ of the code $C= \pi_{T(\Fq)}(\Fq<U>)\subseteq \Fq^{\vert {T(\Fq)}\vert}$, the bound for $t(U)$ is based on the on the minimum distance $d'$ of the dual code $C'=\pi_{T(\Fq)}(\Fq<-H\setminus -U> \subseteq \Fq^{\vert {T(\Fq)}\vert}$, using \refprop{orto} to represent the dual code as an evaluation code. 

The codes have length $\vert {T(\Fq)}\vert$, hence,
\begin{eqnarray*}
r(U) \geq  \vert {T(\Fq)}\vert-d+2=
(\mathrm{the\  maximum\ number\ of\ zeros\ of\ zeros\ of\ }+2 \pi_{T(\Fq)}(f))\\
t(U) \leq d'-2 = \vert {T(\Fq)}\vert-(\mathrm{the\ maximum\ number\ of\ zeros\ of\ } \pi_{T(\Fq)}(g))-2\ .
\end{eqnarray*}

The results follow from the construction of Massey \cite[Section 4.1]{MR2017562}.
\end{proof} 

Of interest is to consider the \emph{coset distance} that is greater than or equal to the minimum distance, which has been used in \cite{MR2588125} to estimate the parameters of secret sharing schemes coming from Algebraic-Geometry codes.

\begin{theorem}\label{strong} Let $U \subseteq H \subset M$ and let $U+U = \{u_1+u_2\vert\ u_1, u_2 \in U \}$ be the Minkowski sum. Let
\begin{eqnarray*}
\pi_{T(\Fq)}:\Fq<U+U>&\rightarrow & \Fq^{\vert {T(\Fq)}\vert}\\
 h&\mapsto&\pi_{T(\Fq)}(h)=(h(P))_{P\in {T(\Fq)}}\ .
\end{eqnarray*}

The linear secret sharing schemes $\mathcal{M}(U)$ of \refdefn{LSSS}
with $n=(q-1)^r-1$ players, has strong multiplication with respect to $\mathcal{A}_{t,n}$ for $t \leq t(U)$, where $t(U)$ is the adversary threshold of $\mathcal{M}(U)$, if
\begin{equation*}\label{strongbet}
t \leq n-1 - (\mathrm{the\ maximal\ number\ of\ zeros\ of\ } \pi_{T(\Fq)}(h))
\end{equation*}
for all $h\in \Fq<U+U>$.
\end{theorem}
\begin{proof}
For $A \in \mathcal{A}_{t,n}$, let $B:=T(\Fq) \setminus (\{P_0\} \cup A)$ with $\vert B \vert =n-t$ elements.
For $f, g \in \Fq<U>$, we have that $f \cdot g \in \Fq<U+U>$. Consider the linear morphism
\begin{eqnarray}\label{multi}
\pi_{B}:\Fq<U+U>&\rightarrow & \Fq^{\vert B\vert}\\
 h&\mapsto&(h(P))_{P\in {B}}
\end{eqnarray}
evaluating at the points in $B$.

By assumption $h \in \Fq<U+U>$ can have at most $n-t-1 < n-t = \vert B \vert$ zeros, therefore $h$ cannot vanish identically on $B$, and we conclude that $\pi_B$ is injective. Consequently, the products  $f(P)\cdot g(P)$ of the shares $P \in B$ determine the product of the secrets $f(P_0)\cdot g(P_0)$, and the scheme has strong multiplication by definition.
\end{proof}

To determine the product of the secrets from the product of the shares amounts to decoding the linear code obtained as the image in (\ref{multi}).

\section{Toric surfaces and linear secret sharing schemes with strong multiplication}

Let $M \simeq \Z^2$ be a 2-dimensional lattice  and assume that $U= M_{\mathbb R} \cap \square$ consists of the integral points of a 2-dimensional integral convex polytope $\square$ in $M_{\mathbb R}= M \otimes_{\mathbb Z}{\mathbb R}$. Let
$N=\Hom_\mathbb Z(M,\mathbb Z)$ be the dual lattice with canonical $\mathbb Z$-bilinear pairing $ <\quad,\quad>: M \times N \rightarrow \mathbb Z.$ 

The support function
$
h_{\square}: N_{\mathbb R} \rightarrow \mathbb R
$
is defined as
$
h_{\square}(n):= {\rm inf}\{<m,n> | \, m \in \square\}
$
and the polytope $\square$ can be reconstructed from the support function
\begin{equation*}
\square_{h} = \{m \in M |\, <m,n> \,\geq \,h(n) \quad \forall n \in N \}.
\end{equation*}

The {\it normal fan} $\Delta$ is the coarsest  fan so that $h_{\square}$ is linear
on each $\sigma \in \Delta$, i.e., for all $\sigma \in \Delta$ there exists  $l_{\sigma} \in M$ so that
\begin{equation*}
h_{\square}(n) = <l_{\sigma},n> \quad \forall n \in \sigma.
\label{linear}
\end{equation*}
Upon refinement of the normal fan, we can assume that two
successive pairs of $n(\rho)$'s generate the lattice and we obtain
{\it the refined normal fan}.
The 1-dimensional cones $\rho \in \Delta$ are generated by unique primitive elements
$n(\rho) \in N \cap \rho$ so that
$\rho =\mathbb R_{\geq 0}\ n(\rho)$.

Let $k=\overline{\mathbb F_q}$ be an algebraic closure of $\mathbb F_q$.

The 2-dimensional {\it algebraic torus} $T_N \simeq 
k^* \times k^* $ is defined by $T_N:= \Hom_\mathbb Z(M,k^*)$. The multiplicative
character $\mathbf e(m)$ for $m \in M$ is the homomorphism 
\begin{eqnarray*}
e(m): T_N&\rightarrow& k^*\\
t&\mapsto& t(m)
\end{eqnarray*} 

Specifically, if $\{n_1,n_2\}$ and $\{m_1,m_2\}$ are dual $\mathbb Z$-bases of $N$ and $M$ and we denote
$u_j:= \mathbf e(m_j),\,j=1,2$, then we have an isomorphism $T_N \simeq k^* \times k^* $ sending $t$ to $(u_1(t),u_2(t))$.
For $m=\lambda_1 m_1 +\lambda_2 m_2$ we have
\begin{equation*}\label{e}
\mathbf e(m)(t)=u_1(t)^{\lambda_1}u_2(t)^{\lambda_2}.
\end{equation*}

The orbits of this action are in one-to-one correspondence with
$\Delta$. For each $\sigma \in \Delta$ let
\begin{equation*}
\orb(\sigma):=\{u:M\cap \sigma \rightarrow k^* | u \text{
is a group homomorphism}\}\ .
\end{equation*}
Define $V(\sigma)$ to be the closure of
$\orb(\sigma)$ in $X_{\square}$.

A $\Delta$-linear support function $h$ gives rise to a polytope $\square$ and an associated Cartier
divisor 
\begin{equation*}\label{Cartier}
D_h=D_{\square}:= -\sum_{\rho \in \Delta (1)} h(n(\rho))\,V(\rho)\ ,
\end{equation*}
where $\Delta (1)$ consists of the 1-dimensional cones in $\Delta$.
In particular
\begin{equation*}
D_m={\rm div}(\mathbf e(-m)), \quad m \in M.
\end{equation*}

\begin{lemma} \label{cohomology} Let  $h$ be a $\Delta$-linear support function
with associated convex polytope $\square$ and Cartier divisor $D_h=D_{\square}$.

The vector space
${\rm H}^0(X,\it O_X(D_h))$ of global sections of $O_X(D_{\square})$, i.e., rational functions $f$ on
$X_{\square}$ so that
${\rm div}(f) + D_{\square} \geq 0$ has dimension $\vert (M \cap \square)\vert$, that is the number af lattice points in $\square$, and has
\begin{equation*}
\{\mathbf e(m) | m \in M \cap \square = U\}
\end{equation*}
as a basis.
\end{lemma}

For a $\Delta$-linear support function $h$ and a 1-dimensional
cone $\rho \in \Delta (1)$ the intersection
number $(D_h;V(\rho))$ between the Cartier divisor $D_h$ of (\ref{Cartier}) and
$V(\rho)) =\mathbb P^1$ is obtained in \cite[Lemma
2.11]{Oda1988}. The 1-dimensional cone $\rho \in \Delta (1)$ is the common face of two 2-dimensional
cones $\sigma', \sigma'' \in \Delta (2)$. Choose primitive
elements $n', n'' \in N$ so that
\begin{align*}
n'+n''& \in \mathbb R \rho\\
\sigma' + \mathbb R \rho &= \mathbb R_{\geq 0} n' + \mathbb R \rho\\
\sigma'' + \mathbb R \rho &= \mathbb R_{\geq 0} n'' + \mathbb R \rho
\end{align*}
\begin{lemma}
\label{inter}For any $l_{\rho} \in M$, such that $h$ coincides
with $l_{\rho}$ on $\rho$, let $\overline{h} = h-l_{\rho}$. Then
\begin{equation*}
(D_h;V(\rho))= -(\overline{h}(n')+\overline{h}(n'')).
\end{equation*}
\end{lemma}
In the 2-dimensional non-singular case let $n(\rho)$ be a
primitive generator for the 1-dimensional cone $\rho$. There
exists an integer $a$ such that
\begin{equation*}
n'+n''+a n(\rho)=0,
\end{equation*}
$V(\rho)$ is itself a Cartier divisor and the above determines the
self-intersection number
\begin{equation*}\label{a}
(V(\rho);V(\rho))=a\ .
\end{equation*}

\subsection{Hirzebruch surfaces}
Let $d, e, r$ be  positive integers and let $\square$ be the polytope
in $M_{\mathbb R}$ with vertices $(0,0),(d,0),(d,e+rd),(0,e)$  rendered in 
Figure \ref{hirzpolytop} and with refined normal fan depicted in Figure \ref{fan2}. The related toric surface is called a \emph{Hirzebruch surface}.

We obtain the following result as a consequence of \reftheorem{thresholds} and the bounds obtained in \cite{53ee7c6020b511dcbee902004c4f4f50} on the number of zeros of functions on such surfaces.
\begin{figure}
\caption{Hirzebruch surfaces. The convex polytope $H$ with vertices $(0,0), (q-2,0), (q-2,q-2), (0,q-2)$, the convex polytope $\square$ with vertices $(0,0), (d,0), (d,e+rd), (0,e)$ and their opposite convex polytopes $-H$ and $-\square$. Also the (non-convex) polytope $-H\setminus -\square$ is depicted.}\label{hirzpolytop}
\begin{center}
\begin{tikzpicture}[scale=0.6]
\draw[very thick] (0,7.5)--(15,7.5);
\draw[very thick] (7.5,0)--(7.5,15);
\draw[thick,fill=black!30] (7.5,7.5)--(9.5,7.5)--(9.5,13.5)--(7.5,10.5)--(7.5,7.5);
\draw[thick,fill=black!20] (7.5,7.5)--(5.5,7.5)--(5.5,1.5)--(7.5,4.5)--(7.5,7.5);
\draw[thick,fill=black!10] (5.5,7.5)--(5.5,1.5)--(7.5,4.5)--(7.5,0)--(0,0)--(0,7.5)--(5.5,7.5);
\draw[thick] (7.5,7.5)--(15,7.5)--(15,15)--(7.5,15)--(7.5,7.5);
\draw (9.5,7.5) node[anchor=north]{$d$};
\draw (7.5,10.5) node[anchor=east]{$e$};
\draw (7.5,13.5) node[anchor=east]{$e+rd$};
\draw(8.5,9.5) node[anchor=north]{$\square$};
\draw(13.5,10.5) node[anchor=north]{$H$};
\draw(6.5,6.5) node[anchor=north]{$-\square$};
\draw(2.5,3.5) node[anchor=north]{$-H\setminus-\square$};
\draw[step=.5cm,gray,thin,dashed](0,0) grid (15,15);
\draw (7.5,15) node[anchor=east]{$q-2$};
\draw (7.5,0) node[anchor=west]{$-(q-2)$};
\draw (14.5,6.5) node[anchor=north,rotate=90]{$(q-2)$};
\draw (0.5,8.5) node[anchor=south,rotate=90]{$-(q-2)$};
\end{tikzpicture}
\end{center}
\end{figure}
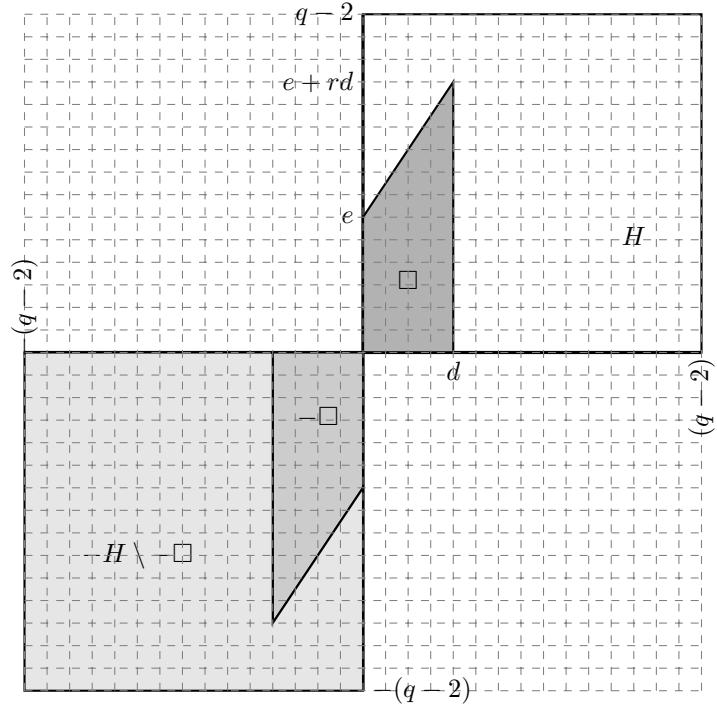

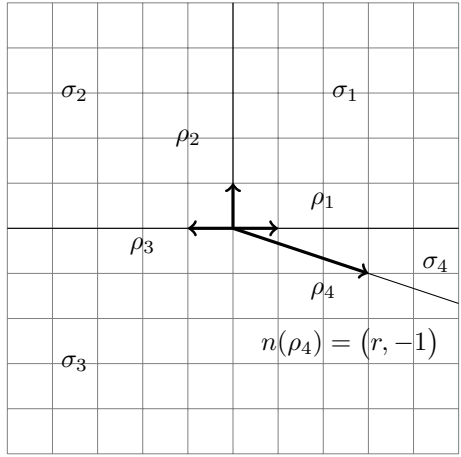
\begin{figure}
\begin{center}
\begin{tikzpicture}[scale=0.6]
\draw[step=1cm,gray,very thin] (-5,-5) grid (5,5);
\draw (-5,0)--(5,0);
\draw (0,0)--(0,5);
\draw (0,0)--(5,-5/3);
\draw[->, very thick] (0,0)--(0,1);
\draw[->, very thick] (0,0)--(-1,0);
\draw[->, very thick] (0,0)--(1,0);
\draw[->, very thick] (0,0)--(3,-1);
\draw (-0.5,2) node[anchor=east] {$\rho_2$};
\draw (-2,0) node[anchor=north] {$\rho_3$};
\draw (2,1) node[anchor=north] {$\rho_1$};
\draw (2,-1) node[anchor=north] {$\rho_4$};
\draw (3,3) node[anchor=east] {$\sigma_1$};
\draw (-3,3) node[anchor=east] {$\sigma_2$};
\draw (-3,-3) node[anchor=east] {$\sigma_3$};
\draw (5,-.8) node[anchor=east] {$\sigma_4$};
\draw (2.6,-2) node[anchor=north] {$n(\rho_4)=\big(r,-1\big)$};

\end{tikzpicture}
\end{center}
\caption{The normal fan and its 1-dimensional cones $\rho_i$, with primitive generators $n(\rho_i)$, and 2-dimensional cones $\sigma_i$ for $i=1,\dots, 4$ of the polytope $\square$ in
Figure \ref{hirzpolytop}.}\label{fan2}
\end{figure}
\begin{theorem} \label{saetning4} 
Let $\square$ be the polytope
in $M_{\mathbb R}$ with vertices $(0,0),(d,0),(d,e+rd),(0,e)$. Assume that  $d\leq q-2$,  $e\leq q-2$ and that $e+rd\leq q-2$. Let $U= M \cap \square$ be the lattice points in $\square$.

Let $\mathcal{M}(U)$ be the linear secret sharing schemes of \refdefn{LSSS} with \emph{support} $T(\mathbb F_q)$ and $(q-1)^2-1$ players.

Then the number of lattice points in $\square$ is
\begin{equation*}
\vert U \vert = \vert (M \cap
\square)\vert =(d+1)(e+1)+r\frac{d(d+1)}{2}\ .
\end{equation*}

The maximal number of zeros of a function $f\in \Fq<U>$ on $T(\Fq)$ is
\begin{equation*}
\max\{d(q-1)+(q-1-d)e,(q-1)(e+dr)\}
\end{equation*}
and the reconstruction threshold as defined in \refdefn{defthresholds} of $\mathcal{M}(U)$  is
\begin{equation*}\label{rU}
r(U) = 1 + \max\{d(q-1)+(q-1-d)e,(q-1)(e+dr)\} \ .
\end{equation*}
\end{theorem}

\begin{remark}
The polytope $-H \setminus -U$ is not convex, so our method using intersection theory does not determine the privacy threshold $t(U)$. It would be interesting to examine the methods and results of \cite{MR2272243}, \cite{MR2476837}, \cite{MR2322944}, \cite{MR2360532}, \cite{Beelen},\cite{MR3093852} \cite{MR3345095},  and \cite{DBLP:journals/corr/Little15} for toric codes in this context.
\end{remark}

\subsection{Toric surfaces with associated linear secret sharing schemes with strong multiplication}
Let $a, b$ be  positive integers $0 \leq b \leq a \leq q-2$, and let $\square$ be the polytope
in $M_{\mathbb R}$ with vertices $(0,0),(a,0),(b,q-2),(0,q-2)$  rendered in 
Figure \ref{polytope} and with normal fan depicted in Figure \ref{cones}. 

Under these assumptions the polytopes $\square$, $-H\setminus-\square$ and $\square+\square$ are convex and we can use intersection theory on the associated toric surface to bound the number of zeros of functions and thresholds.

The primitive generators of the 1-dimensional cones are
\begin{equation*}
n(\rho_1)=\begin{pmatrix}1\\0\end{pmatrix},\ 
n(\rho_2)=\begin{pmatrix}0\\1\end{pmatrix},\ 
n(\rho_3)=\begin{pmatrix}\frac{-(q-2)}{\gcd(a-b,q-2)}\\\frac{-(a-b)}{\gcd(a-b,q-2)}\end{pmatrix}\ ,
n(\rho_4)=\begin{pmatrix}0\\-1\end{pmatrix}\ .
\end{equation*} 

For $i=1,\dots, 4$, the 2-dimensional cones $\sigma_i$  are shown in Figure \ref{cones}. The faces of $\sigma_1$ are $\{\rho_1, \rho_2\}$, the faces of $\sigma_2$ are $\{\rho_2, \rho_3\}$, the faces of $\sigma_3$ are $\{\rho_3, \rho_4\}$ and the faces of $\sigma_4$ are $\{\rho_4, \rho_1\}$.

The support function of $\square$ is:
\begin{equation}\label{supportfct}
h_{\square}\begin{pmatrix}n_1\\n_2\end{pmatrix}=
\begin{cases}
\begin{pmatrix}0\\0\end{pmatrix} . \begin{pmatrix}n_1\\n_2 \end{pmatrix} & \text{if $\begin{pmatrix}n_1\\n_2\end{pmatrix} \in \sigma_1$},\\
\begin{pmatrix}a\\0\end{pmatrix} . \begin{pmatrix}n_1\\n_2 \end{pmatrix} & \text{if $\begin{pmatrix}n_1\\n_2\end{pmatrix} \in \sigma_2$} ,\\
\begin{pmatrix}b\\q-2\end{pmatrix} . \begin{pmatrix}n_1\\n_2 \end{pmatrix} & \text{if $\begin{pmatrix}n_1\\n_2\end{pmatrix} \in \sigma_3$} ,\\
\begin{pmatrix}0\\q-2\end{pmatrix} . \begin{pmatrix}n_1\\n_2 \end{pmatrix} & \text{if $\begin{pmatrix}n_1\\n_2\end{pmatrix} \in \sigma_4$}.
\end{cases}
\end{equation}

 The related toric surface is in general singular as $\{n(\rho_2), n(\rho_3)\}$ and $\{n(\rho_3), n(\rho_4)\}$ are not bases for the lattice $M$. We can desingularize by subdividing the cones $\sigma_2$ and $\sigma_3$, however, our calculations will only involve the cones $\sigma_1$ and $\sigma_2$, so we refrain from that.

For all pairs of  1-dimensional cones $\rho_i, \rho_j \in \Delta (1), i= 1, \dots , 4$, the intersection numbers $(V(\rho_i);V(\rho_j))$  are determined by the methods above, however, we only need the self-intersection number $(V(\rho_1);V(\rho_1))$, and as
\begin{equation*}
n(\rho_2)+n(\rho_4) + 0 \cdot n(\rho_1)=0\ ,
\end{equation*}
we have that
\begin{equation}\label{self1}
(V(\rho_1);V(\rho_1)) = 0
\end{equation}
by the remark following \reflemma{inter}.

\begin{theorem} \label{saetning5} 

Assume $a, b$ are integers with $0 \leq b \leq a \leq q-2$.

Let $\square$ be the polytope
in $M_{\mathbb R}$ with vertices $(0,0),(a,0),(b,q-2),(0,q-2)$  rendered in 
Figure \ref{polytope}, and let $U= M  \cap \square$ be the lattice points in $\square$.

Let $\mathcal{M}(U)$ be the linear secret sharing schemes  \refdefn{LSSS} with \emph{support} $T(\mathbb F_q)$ and $n=(q-1)^2-1$ players.

\begin{itemize}

\item[\it{i)}]The maximal number of zeros of $\pi_{T(\Fq)}(f)$ for $f\in \Fq<U>$  is less than or equal to 
\begin{equation*}
(q-1)^2-(q-1-a) \ .
\end{equation*}

\item[\it{ii)}] The reconstruction threshold as defined in \refdefn{defthresholds} satisfies
\begin{equation*}\label{rU2}
r(U) \leq 1+ (q-1)^2-(q-1-a) \ .
\end{equation*}

\item[\it{iii)}] The privacy threshold as defined in \refdefn{defthresholds} satisfies
\begin{equation*}\label{tU2}
 t(U) \geq b-1\ .
\end{equation*}

\item[\it{iv)}] Assume $2a\leq q-2$. The secret sharing scheme has $t$-strong multiplication for 
\begin{equation*}
t\leq \min\{b-1,(q-2-2a)-1\}\ .
\end{equation*}

\end{itemize}
\end{theorem}

\begin{proof}
Let $m_{1}=(1,0)$. 
The $\Fq$-rational points of $T \simeq \Fbarq^{*} \times \Fbarq^{*}$
belong to the $q-1$ lines on $X_{\square}$ given by
\begin{equation*}
\prod_{\eta \in \F}(\mathbf e(m_{1})-\eta) =0\ .
\end{equation*}
Let $0 \neq f \in {\rm H}^0(X,\it O_X(D_h))$. Assume that $f$ is zero along precisely $c$ of these lines. 

As $\mathbf e(m_{1})-\eta$ and $\mathbf e(m_{1})$ have the same divisors of poles, they have equivalent
divisors of zeroes, so
\begin{equation*}
(\mathbf
e(m_{1})-\eta)_{0} \sim (\mathbf
e(m_{1}))_{0}\ .
\end{equation*}
Therefore
\begin{equation*}
{\rm div}(f) + D_h -c (\mathbf
e(m_{1}))_{0}\geq 0
\end{equation*}
or equivalently
\begin{equation*}
f \in {\rm H}^0(X,\it
O_X(D_h-c (\mathbf
e(m_{1}))_{0})\ .
\end{equation*}
This implies that $c \leq a$ according to \reflemma{cohomology}. 

On any of the other $q-1-c$ lines the number of zeroes of $f$ is at most the intersection number
\begin{equation*}
(D_{h}-c (\mathbf e(m_{1}))_{0};(\mathbf
e(m_{1}))_{0})\ .
\end{equation*}
 
This number can be calculated using \reflemma{inter} using the observation that $(\mathbf e(m_{1}))_{0}=V(\rho_1)$.

We get from (\ref{supportfct}) and (\ref{self1}) that
\begin{eqnarray*}
&(D_{h}-c (\mathbf e(m_{1}))_{0};(\mathbf
e(m_{1}))_{0})= \\
&(D_{h};(\mathbf
e(m_{1}))_{0}) - c (\mathbf e(m_{1}))_{0};(\mathbf
e(m_{1}))_{0})= \\
&- h_{\square}\begin{pmatrix}0\\1\end{pmatrix}-h_{\square}\begin{pmatrix}0\\-1\end{pmatrix} = q-2\ ,
\end{eqnarray*}
as $l_{{\rho}_1} = \begin{pmatrix}0\\0\end{pmatrix} \in M$. 

As $0\leq c \leq a$, we conclude the
total number of zeroes for $f$ is at most
\begin{equation*}
c(q-1)+(q-1-c)(q-2) \leq a(q-1)+(q-1-a)(q-2)= (q-1)^2-(q-1-a)
\end{equation*}
proving {\it i)}.

According to \reftheorem{thresholds}, we have the inequality of {\it ii)}
\begin{equation*}
r(U) \leq 1+ (q-1)^2-(q-1-a) \ .
\end{equation*}

We obtain {\it iii)} by using the result in {\it i)} on the polytope $(q-2,q-2)+\big(-H \setminus -\square \big)$ with vertices $(0,0), (q-2-b,0),(q-2-a,q-2)$ and $(q-2,q-2)$. The maximum number of zeros of $\pi_{T(\Fq)}(g)$ for $g\in \Fq<-H \setminus - U>$ is by \reflemma{trans} and the result in {\it i)} less than or equal to $(q-1)^2-(q-1-(q-2-b))=(q-1)^2-1-b$ and {\it iii)} follows from \reftheorem{thresholds}.

To prove {\it iv)} assume $t \leq (q-2-2a)-1$ and $t \leq b-2$. We will use \reftheorem{strong}.

Consider the Minkowski sum $U+U$ and let $V=\overline{U + U}$ be  its reduction modulo $q-1$ as in \refcor{reduction}. Under the assumption $2a \leq q-2$, we have that $V=\overline{U + U}$ is the lattice points of the integral convex polytope  with vertices $(0,0),(2a,0), (2b, q-2)$ and $(0,q-2)$. 

By the result in {\it i)} the maksimum number of zeros of $\pi_{T(\Fq)}(h)$ for $h\in \Fq<V>$
is less than or equal to $(q-1)^2-(q-1-2a)$. As the number of players is $n=(q-1)^2-1$, the right hand side of the condition (\ref{strongbet}) of \reftheorem{strong} is at least $(q-2-2a)-1$, which by assumption is at least $t$.

By assumption $t \leq b-1$ and from {\it iii)} we have that $b-1 \leq t(U)$. We conclude that $t \leq t(U)$.

\end{proof}

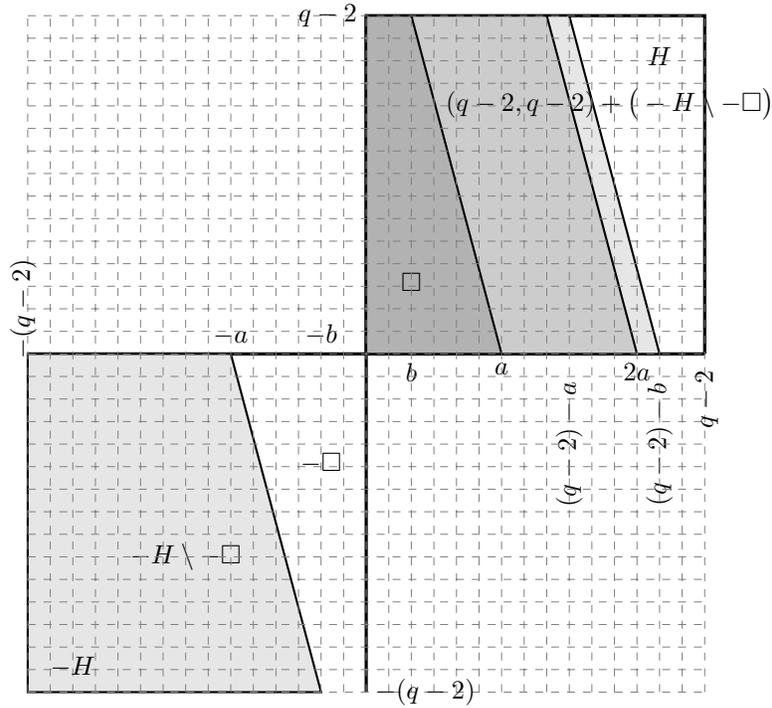
\begin{figure}[H]
\caption{The convex polytope $H$ with vertices $(0,0), (q-2,0), (q-2,q-2), (0,q-2)$ and the convex polytope $\square$ with vertices $(0,0), (a,0), (b,q-2), (0,q-2))$ are shown. Also their opposite convex polytopes $-H$ and $-\square$, the complement
$ -H \setminus -\square$ and its translate $(q-2,q-2)+\big(-H \setminus -\square \big)$ are depicted. 
Finally the convex hull of the reduction modulo $q-1$ of the Minkowski sum $U+U$ of the lattice points $U= \square \cap M$ in $\square$, is rendered. It has vertices $(0,0),(2a,0), (2b,q-2)$ and $(0,q-2)$.}\label{polytope}
\begin{center}
\begin{tikzpicture}[scale=0.6]
\draw[very thick] (0,7.5)--(15,7.5);
\draw[very thick] (15,7.5)--(15,15)--(7.5,15)--(7.5,7.5);
\draw[very thick] (7.5,0)--(7.5,15);
\draw (8.5,7.5) node[anchor=north]{$b$};
\draw (10.5,7.5) node[anchor=north]{$a$};
\draw(8.5,9.5) node[anchor=north]{$\square$};
\draw(14,14.5) node[anchor=north]{$H$};
\draw (4.5,7.5) node[anchor=south]{$-a$};
\draw (6.5,7.5) node[anchor=south]{$-b$};
\draw(6.5,5.5) node[anchor=north]{$-\square$};
\draw[thick,fill=black!10] (4.5,7.5)--(0,7.5)--(0,0)--(6.5,0)--(4.5,7.5);
\draw(3.5,3.5) node[anchor=north]{$-H \setminus -\square$};
\draw[thick,fill=black!10] (12,15)--(7.5,15)--(7.5,7.5)--(14,7.5)--(12,15);
\draw[thick,fill=black!20] (7.5,7.5)--(13.5,7.5)--(11.5,15)--(7.5,15)--(7.5,7.5);
\draw[thick,fill=black!30] (7.5,7.5)--(10.5,7.5)--(8.5,15)--(7.5,15)--(7.5,7.5);
\draw(8.5,9.5) node[anchor=north]{$\square$};
\draw(12.9,13.6) node[anchor=north]{$(q-2,q-2)+\big(-H \setminus -\square \big)$};
\draw (13.5,5.5) node[anchor=north,rotate=90]{$(q-2)-b$};
\draw (11.5,5.5) node[anchor=north,rotate=90]{$(q-2)-a$};
\draw (13.5,7.5) node[anchor=north]{$2a$};
\draw[step=.5cm,gray,very thin,dashed]
(0,0) grid (15,15);
\draw(1,1) node[anchor=north]{$-H$};
\draw (7.5,15) node[anchor=east]{$q-2$};
\draw (7.5,0) node[anchor=west]{$-(q-2)$};
\draw (14.6,6.5) node[anchor=north,rotate=90]{$q-2$};
\draw (0.4,8.5) node[anchor=south,rotate=90]{$-(q-2)$};
\end{tikzpicture}
\end{center}
\end{figure}

\begin{figure}[H]
\caption{The normal fan and its 1-dimensional cones $\rho_i$, with primitive generators $n(\rho_i)$, and 2-dimensional cones $\sigma_i$ for $i=1,\dots, 4$ of the polytope $\square$ in Figure \ref{polytope}.}\label{cones}
\begin{center}
\begin{tikzpicture}[scale=0.6]
\draw[thick] (0,7.5)--(15,7.5);
\draw[thick] (7.5,0)--(7.5,15);
\draw[very thick] (7.5,7.5)--(7.5,15);
\draw[very thick] (7.5,7.5)--(7.5,0);
\draw[very thick] (7.5,7.5)--(15,7.5);
\draw[very thick] (7.5,7.5)--(0,5.5);
\draw[->, very thick] (7.5,7.5)--(7.5,8);
\draw[->, very thick] (7.5,7.5)--(7.5,7);
\draw[->, very thick] (7.5,7.5)--(8,7.5);
\draw[->, very thick] (7.5,7.5)--(0,5.5);
\draw (7.3,9.5) node[anchor=east] {$\rho_2$};
\draw (2,7) node[anchor=north] {$\rho_3$};
\draw (9.5,8.3) node[anchor=north] {$\rho_1$};
\draw (8,5.8) node[anchor=north] {$\rho_4$};

\draw (11,10.5) node[anchor=east] {$\sigma_1$};
\draw (5,10.5) node[anchor=east] {$\sigma_2$};
\draw (11,4.5) node[anchor=east] {$\sigma_4$};
\draw (5,4.5) node[anchor=east] {$\sigma_3$};
\draw (3.6,3) node[anchor=north] {$n(\rho_3)=\big(\frac{-(q-2)}{\gcd(a-b,q-2)},\frac{-(a-b)}{\gcd(a-b,q-2)}\big)$};
\draw[step=.5cm,gray,very thin,dashed](0,0) grid (15,15);
\end{tikzpicture}
\end{center}
\end{figure}
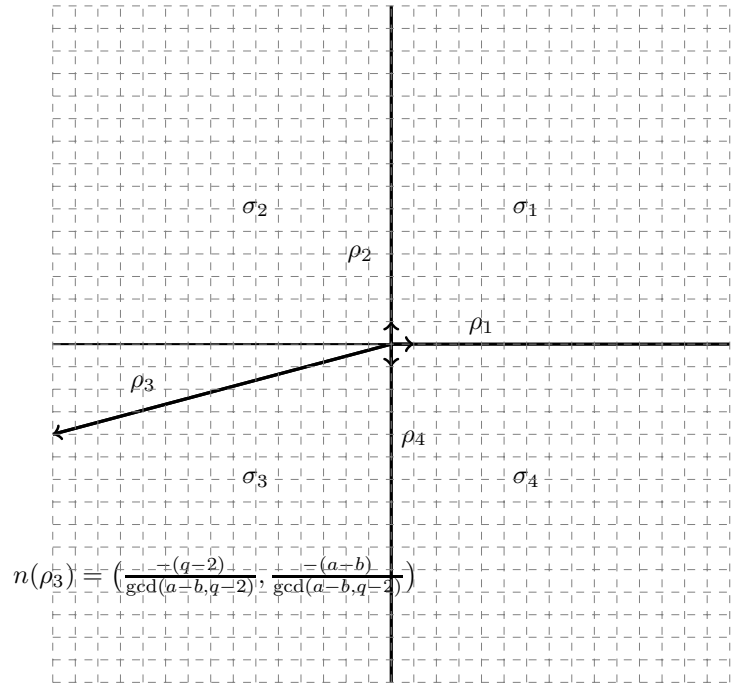

\bibliographystyle{amsplain}
\bibliography{SSS}

\end{document}